\documentclass[11pt]{article}
\usepackage[a4paper,margin=1in]{geometry}
\usepackage[T1]{fontenc}
\usepackage[utf8]{inputenc}
\usepackage{lmodern}
\usepackage{amsmath,amssymb,amsthm,mathtools}
\usepackage{microtype}
\usepackage{enumitem}
\usepackage{booktabs}
\usepackage{float}
\usepackage{hyperref}
\usepackage{xcolor}
\usepackage{verbatim}
\hypersetup{colorlinks=true,linkcolor=blue!50!black,citecolor=blue!50!black,urlcolor=blue!50!black}

\newtheorem{theorem}{Theorem}[section]
\newtheorem{proposition}[theorem]{Proposition}
\newtheorem{lemma}[theorem]{Lemma}
\newtheorem{corollary}[theorem]{Corollary}
\newtheorem{conjecture}[theorem]{Conjecture}
\newtheorem{problem}[theorem]{Problem}

\theoremstyle{definition}
\newtheorem{definition}[theorem]{Definition}

\newtheorem{remark}[theorem]{Remark}

\newcommand{\BWT}{\operatorname{BWT}}
\newcommand{\runs}{\rho}
\newcommand{\cruns}{\rho_{\circ}}

\newcommand{\Z}{\mathbb Z}

\newcommand{\sort}{\operatorname{sort}}
\newcommand{\dH}{d_{\mathrm H}}
\newcommand{\Parikh}{\operatorname{Parikh}}

\newcommand{\Cyc}{\operatorname{Cyc}}

\newcommand{\GRR}{\operatorname{GRR}}

\title{An extremal problem for completely unclustered Burrows--Wheeler images}
\author{David Kumallagov}
\date{\today}

\begin{document}
\maketitle

\begin{abstract}
The Burrows--Wheeler transform is usually viewed as a clustering transform:
it tends to group equal letters into long runs.  We study the opposite
extremal regime, where the BWT output is completely unclustered, that is, has
as many equal-letter runs as positions.  Known results imply, on the one hand,
that the number of runs in the BWT of a Lyndon word can increase by at most a
factor of two, and, on the other hand, that over every alphabet of size at
least three completely unclustered BWT images exist in every length. This leads to the extremal problem lying between these two facts.  For \(k\ge3\), let \(U_k(n)\) be the minimum cyclic run number of a primitive necklace of length \(n\) whose BWT has \(n\) runs.

We prove the universal lower bound \(U_k(n)\ge\lceil n/2\rceil\), reduce the
sharpness problem for one-cycle BWT images \(L\) to the Hamming identity
\[
   \cruns(\BWT^{-1}(L))=\dH(L,\sort(L)),
\]
and develop a natural multiset-of-necklaces relaxation with an explicit
constant-cycle correction.  We compute the small values, including the
exceptional value \(U_k(6)=4\), prove a parity obstruction for the Parikh
vectors of sharp examples, and determine the multiset relaxation exactly.
Finally, for every prime \(p\equiv5\pmod8\) for which \(2\) is a primitive
root modulo \(p\), we prove sharpness in the adjacent lengths \(p-1\) and
\(p\).  Under the corresponding Artin-type infinitude hypothesis, this gives
infinitely many adjacent sharp pairs.
\end{abstract}

\tableofcontents

\section{Introduction and motivation}

Run-length compression of the Burrows--Wheeler transform
(BWT)~\cite{BurrowsWheeler1994} exploits a basic clustering effect: after
sorting contexts, equal letters often become grouped into long runs.  This
paper studies the opposite extremal regime, in which the BWT output is
completely unclustered.

A factor-two theorem of Mantaci et al.~\cite{MRRSV2017} shows that this
opposite behaviour is universally bounded: for every Lyndon word \(w\),
\begin{equation}\label{eq:MRRSV-bound}
  \runs(\BWT(w))\le 2\runs(w).
\end{equation}
Moreover, the factor \(2\) is sharp, and the binary case of maximal
unclustering is governed by primitive-root arithmetic.

A recent theorem of Fici et al.~\cite{FGRS2025} gives a complementary
existence result.  For every integer \(n>0\) and every alphabet size
\(k\ge3\), there exists a primitive necklace \([u]\) of length \(n\) over
\(\Sigma_k\) whose BWT is completely unclustered:
\[
  \runs(\BWT([u]))=n.
\]
Their proof proceeds through generalized de Bruijn words and then uses
insertion/deletion operations preserving the cyclicity of the standard
permutation.  For the generalized de Bruijn viewpoint on BWT and standard
permutations, see also Fici--Gabory~\cite{FiciGabory2025}.

These two results leave a natural joint extremal problem.
Inequality~\eqref{eq:MRRSV-bound} implies that if \(\BWT([u])\) has \(n\)
runs, then the input necklace must have at least \(\lceil n/2\rceil\)
cyclic runs.  The existence theorem of~\cite{FGRS2025} shows that BWT outputs
with \(n\) runs exist in every length over non-binary alphabets.  The missing
question is whether the two extremal behaviours can occur simultaneously:
\[
   \runs(\BWT([u]))=n,
   \qquad
   \cruns([u])\approx n/2.
\]
We investigate this extremal question in the present paper.  

Our main contributions are:
\begin{enumerate}[label=\rm(\roman*)]
\item we introduce \(U_k(n)\) and prove the universal lower bound
      \(U_k(n)\ge\lceil n/2\rceil\);
\item we reduce the one-cycle sharpness problem to an explicit Hamming-distance criterion;
\item we compute the first small values, including the exceptional value
      \(U_k(6)=4\);
\item we construct Artin-type sharp families in the adjacent lengths
      \(p-1\) and \(p\);
\item we prove a parity obstruction for the Parikh vectors of sharp examples;
\item we determine the associated multiset-of-necklaces relaxation exactly.
\end{enumerate}

\section{Preliminaries}

\subsection{Words, necklaces, and runs}

Fix an ordered alphabet
\[
  \Sigma_k=\{0<1<\cdots<k-1\}.
\]
For a word \(w=w_0w_1\cdots w_{n-1}\in\Sigma_k^n\), its ordinary, linear number of equal-letter runs is
\[
  \runs(w)=
  \begin{cases}
    0,& n=0,\\
    1+\#\{0\le i<n-1: w_i\ne w_{i+1}\},& n>0.
  \end{cases}
\]

Throughout the paper, positions in words of length \(n\) are indexed from
\(0\) to \(n-1\).  For a letter \(a\in\Sigma_k\), we write
\[
   |w|_a=\#\{0\le i<n:w_i=a\}
\]
for the number of occurrences of \(a\) in \(w\).  The Parikh vector of \(w\) is
\[
   \Parikh(w)=(|w|_0,\ldots,|w|_{k-1}).
\]

A necklace \([w]\) is a conjugacy class of words under cyclic rotation.  It is primitive, or aperiodic, if it contains \(n\) distinct conjugates.  For a nonconstant necklace we define its cyclic run number by
\[
  \cruns([w])=
  \#\{i\in\Z/n\Z: w_i\ne w_{i+1}\}.
\]
For a constant necklace we set \(\cruns([w])=1\).  Equivalently,
\[
  \cruns([w])=\min\{\runs(v):v\in[w]\}.
\]

\subsection{The BWT of a primitive necklace}

Let \([u]\) be a primitive necklace of length \(n\).  Its BWT matrix is the matrix whose rows are the \(n\) cyclic shifts of \(u\), sorted lexicographically.  The  \(\BWT([u])\) is the last column of this matrix.  The first column is the sorted word
\[
  F=\sort(L),\qquad L=\BWT([u]).
\]

If \(L=L_0L_1\cdots L_{n-1}\), its standard permutation \(\pi_L\) is the stable sorting map from the last column to the first column.  More explicitly, if \(L_i=a\) and this is the \(r\)-th occurrence of \(a\) in \(L\), counted from \(0\), then
\begin{equation}\label{eq:standard-permutation}
  \pi_L(i)=C_a+r,
  \qquad
  C_a=\sum_{b<a}|L|_b.
\end{equation}
Thus \(\pi_L(i)\) is the position in \(F\) corresponding to the occurrence \(L_i\).  In particular,
\begin{equation}\label{eq:L-F-pi}
  L_i=F_{\pi_L(i)}.
\end{equation}
A standard criterion says that a word \(L\in\Sigma_k^n\) is the BWT of an aperiodic necklace if and only if \(\pi_L\) is a single cycle of length \(n\); see, e.g., \cite{GRR2012,FGRS2025,LikhomanovShur2011}.

\begin{definition}
A word \(L\in\Sigma_k^n\) is called completely unclustered if
$\rho(L)=n.$
Equivalently, \(L_i\ne L_{i+1}\) for all \(0\le i<n-1\).
If, in addition, \(\pi_L\) is an \(n\)-cycle, then \(L\) is a completely
unclustered BWT image of one primitive necklace.

\end{definition}

\section{The extremal function \texorpdfstring{$U_k(n)$}{Uk(n)}}

\begin{definition}
For \(k\ge3\) and \(n\ge1\), define
\begin{equation}\label{eq:Ukn-def}
  U_k(n)=\min\left\{\cruns([u]):
  [u]\text{ is a primitive necklace over }\Sigma_k,\ 
  |u|=n,\ 
  \runs(\BWT([u]))=n
  \right\}.
\end{equation}

The function \(U_k(n)\) measures the smallest possible cyclic
run-complexity of an input primitive necklace of length \(n\) whose BWT is
as unclustered as possible. Equivalently, it asks how clustered the input
can still be if the BWT output has the maximum possible number \(n\) of
equal-letter runs. Thus \(U_k(n)=\lceil n/2\rceil\) means that the BWT can
attain the extremal factor-two run amplification.

\end{definition}

For \(k\ge3\), the non-binary existence theorem for completely unclustered
BWT images~\cite{FGRS2025} implies that the set in~\eqref{eq:Ukn-def}
is nonempty for every \(n\).  Thus \(U_k(n)\) is finite for all \(n\) and \(k\ge3\).

\begin{proposition}\label{prop:lower-bound}
For every \(k\ge3\) and every \(n\ge1\),
\begin{equation}\label{eq:lower-bound}
  U_k(n)\ge \left\lceil\frac n2\right\rceil.
\end{equation}

\end{proposition}

\begin{proof}
Let \(w\) be the Lyndon representative of \([u]\). Since BWT is constant on
conjugacy classes, \(\runs(\BWT(w))=n\). By
\cite[Theorem~3.3]{MRRSV2017}, $\runs(\BWT(w))\le 2\runs(w).$
For the Lyndon representative of a primitive necklace, the Lyndon cut occurs at
a cyclic run boundary: unless \(|w|=1\), the first letter of \(w\) is strictly
smaller than its last letter. Hence $\runs(w)=\cruns([u]).$
Therefore \(n\le 2\cruns([u])\), and the claim follows.
\end{proof}

\begin{remark}
For every \(k\ge3\), any ternary example over \(\{0,1,2\}\) is also an example over \(\Sigma_k\).  Hence
\[
   U_k(n)\le U_3(n)\qquad(k\ge3).
\]
Since the lower bound \eqref{eq:lower-bound} is independent of \(k\), proving a sharp ternary upper bound proves the same formula for all larger alphabets.
\end{remark}

\section{Standard-permutation formulae for input and output runs}

The following elementary lemma is the main bookkeeping tool.  It turns the extremal problem into a problem about the crossing pattern of a permutation with respect to the ordered intervals of the sorted word \(F\).

\begin{lemma}\label{lem:run-formulae}
Let \(L\in\Sigma_k^n\), let \(F=\sort(L)\), and let \(\pi=\pi_L\) be the standard permutation.  Suppose that \(\pi\) is a single cycle, let \([u]=\BWT^{-1}(L)\), and assume \(n>1\) for the input-run formula below.  Then:
\begin{enumerate}[label=\rm(\alph*)]
\item
\[
  \runs(L)=1+\#\{0\le i<n-1:F_{\pi(i)}\ne F_{\pi(i+1)}\}.
\]
In particular, \(L\) is completely unclustered if and only if
\[
  F_{\pi(i)}\ne F_{\pi(i+1)}\qquad(0\le i<n-1).
\]
\item
The cyclic input run number is
\begin{equation}\label{eq:input-crossing-formula}
  \cruns([u])=
  \#\{i\in\{0,\ldots,n-1\}:F_i\ne F_{\pi(i)}\}.
\end{equation}

\end{enumerate}
\end{lemma}

\begin{proof}
Part (a) follows from \(L_i=F_{\pi_L(i)}\). Part~(b) is the standard inverse-BWT formula written in the convention where
\(\pi_L^{-1}\) is the FL-mapping
\cite[Proposition~2.1]{MRRSV2017}.
\end{proof}

\begin{definition}[Crossing number]
Let \(F\) be a sorted word and \(\pi\) a permutation of its positions.  Define
\[
  \operatorname{cr}_F(\pi)=\#\{i:F_i\ne F_{\pi(i)}\}.
\]
By Lemma~\ref{lem:run-formulae}, for a BWT image \(L\) with \(F=\sort(L)\),
\[
  \cruns(\BWT^{-1}(L))=\operatorname{cr}_F(\pi_L).
\]
\end{definition}

\begin{proposition}\label{prop:hamming-reduction}
Let \(L\in\Sigma_k^n\), let \(F=\sort(L)\), and suppose that \(n>1\) and that the standard
permutation \(\pi_L\) is an \(n\)-cycle.  Then
\begin{equation}\label{eq:hamming-reduction}
   \cruns(\BWT^{-1}(L))
   =\dH(L,F),
\end{equation}
where \(\dH\) denotes Hamming distance.  
\end{proposition}

\begin{proof}
By the definition of the standard permutation, $L_i=F_{\pi_L(i)}.$ 
Therefore Lemma~\ref{lem:run-formulae} gives
\[
   \cruns(\BWT^{-1}(L))
   =\#\{i:F_i\ne F_{\pi_L(i)}\}
   =\#\{i:F_i\ne L_i\}
   =\dH(F,L).
\]

\end{proof}

\begin{corollary}
\label{cor:hamming-construction-criterion}
Conjecture~\ref{conj:sharp} is equivalent to the following
ternary construction problem: for every \(n\ge7\), construct a word
\(L_n\in\{0,1,2\}^n\) such that 
\begin{enumerate}[label=\rm(\alph*)]
\item \(L_{n,i}\ne L_{n,i+1}\) for all \(0\le i<n-1\);
\item \(\pi_{L_n}\) is an \(n\)-cycle;
\item \(\dH(L_n,\sort(L_n))=\lceil n/2\rceil\).
\end{enumerate}
\end{corollary}

\section{The Gessel--Restivo--Reutenauer viewpoint}

The standard-permutation formulation above is a one-cycle fragment of the
classical Gessel--Restivo--Reutenauer bijection~\cite{GRR2012}, which we call
the GRR bijection below.  Adding this viewpoint is useful for two reasons.  First, it embeds the function \(U_k(n)\) into a broader multiset-of-necklaces framework.  Second, it provides an exact Parikh-vector baseline count for possible one-cycle BWT images before imposing the completely-unclustered condition.

\subsection{The GRR multiset associated with a word}

Let \(L\in\Sigma_k^n\), let \(\pi_L\) be its standard permutation in the convention of \eqref{eq:standard-permutation}, and put
\[
   \tau_L=\pi_L^{-1}.
\]
Write the disjoint cycle decomposition of \(\tau_L\).  If
\[
   C=(i_0,i_1,\ldots,i_{m-1})
\]
is a cycle, where \(\tau_L(i_j)=i_{j+1}\) cyclically, associate to it the necklace
\[
   N_C=[L_{i_0}L_{i_1}\cdots L_{i_{m-1}}].
\]
The GRR map sends
\[
   L\longmapsto \Phi(L):=\{N_C:C\in\Cyc(\tau_L)\},
\]
a multiset of primitive necklaces.  With the above convention one uses the inverse standard permutation; using the standard permutation instead gives the same cycle type and an equivalent convention. The GRR theorem states that \(\Phi\) is a bijection from words of length
\(n\) over the ordered alphabet to multisets of primitive necklaces of total
length \(n\), preserving the Parikh vector and the cycle type.  The usual BWT of one primitive necklace is precisely the inverse map restricted to multisets consisting of a single necklace~\cite{GRR2012}.

\begin{definition}[GRR total run number]
For \(L\in\Sigma_k^n\), define
\[
   R_{\GRR}(L)=\sum_{C\in\Cyc(\tau_L)}\cruns(N_C),
   \qquad \tau_L=\pi_L^{-1}.
\]
\end{definition}

\begin{proposition}
\label{prop:grr-hamming-correction}
Let \(L\in\Sigma_k^n\), let \(F=\sort(L)\), and put \(\tau_L=\pi_L^{-1}\).
A cycle \(C\in\Cyc(\tau_L)\) will be called \emph{letter-constant} if all
letters \(L_i\), \(i\in C\), are equal.  Let
\[
   c_0(L)=\#\{C\in\Cyc(\tau_L): C\text{ is letter-constant}\}.
\]
Then
\begin{equation}\label{eq:grr-hamming-correction}
   R_{\GRR}(L)=\dH(L,F)+c_0(L).
\end{equation}
In particular, in the one-cycle nonconstant case the correction term is zero,
and \eqref{eq:grr-hamming-correction} reduces to
\(R_{\GRR}(L)=\dH(L,\sort(L))\).
\end{proposition}

\begin{proof}
For a cycle \(C\) of \(\tau_L\), the cyclic run number of the associated
necklace is the number of colour changes along the cycle, except in the
letter-constant case, where our convention gives one cyclic run rather than
zero.  Hence
\[
   R_{\GRR}(L)=
   \#\{i:L_i\ne L_{\tau_L(i)}\}+c_0(L).
\]
Since \(L_i=F_{\pi_L(i)}\) and \(\tau_L=\pi_L^{-1}\), we have
\[
   L_i\ne L_{\tau_L(i)}
   \Longleftrightarrow
   F_{\pi_L(i)}\ne F_i.
\]
Therefore
\[
   \#\{i:L_i\ne L_{\tau_L(i)}\}
   =\#\{i:F_i\ne F_{\pi_L(i)}\}
   =\#\{i:F_i\ne L_i\}
   =\dH(F,L),
\]
which proves the formula.
\end{proof}

\begin{remark}
The correction term in Proposition~\ref{prop:grr-hamming-correction} is essential.
For instance, for \(L=01\) one has \(\dH(L,\sort(L))=0\), while the GRR multiset
consists of two one-letter necklaces and hence \(R_{\GRR}(L)=2\).  Thus the
GRR relaxation cannot be analysed by Hamming distance alone; constant cycles
must be counted separately.
\end{remark}

This suggests a relaxed version of the extremal function in which the inverse object is allowed to be a multiset of necklaces rather than one necklace.

\begin{definition}[GRR multiset relaxation]
For \(k\ge3\), define
\[
   \widetilde U_k(n)=
   \min\{R_{\GRR}(L):L\in\Sigma_k^n,\ \runs(L)=n\}.
\]
\end{definition}

Since every completely unclustered one-cycle BWT image \(L\) is also an
admissible word in the GRR relaxation, and in the one-cycle case
\(R_{\GRR}(L)=\cruns(\BWT^{-1}(L))\), we immediately have
\[
   \widetilde U_k(n)\le U_k(n).
\]
Thus the original problem is the one-cycle refinement of a natural GRR multiset minimization problem.

We shall determine \(\widetilde U_k(n)\) exactly in
Theorem~\ref{thm:tilde-U-exact}.  The remaining issue is whether a minimizing
GRR multiset can be represented by a single standard cycle.

Proposition~\ref{prop:grr-hamming-correction} shows that this relaxation is still
not merely a Hamming-distance minimization problem.  Any possible gap between
\(\widetilde U_k(n)\) and \(U_k(n)\) must take into account both crossing edges
and letter-constant GRR cycles.  Thus the multiset relaxation is useful as a
structural comparison, but it does not by itself prove or disprove the sharp
one-cycle conjecture.

\subsection{A Parikh-vector baseline count}

Let
\[
   \alpha=(\alpha_0,\ldots,\alpha_{k-1}),\qquad |\alpha|=\sum_a\alpha_a=n,
\]
be a Parikh vector.  Let \(B_k(\alpha)\) denote the number of words \(L\in\Sigma_k^n\) with Parikh vector \(\alpha\) such that \(\pi_L\) is a single cycle.  Equivalently, \(B_k(\alpha)\) counts BWT images of primitive necklaces with Parikh vector \(\alpha\).  This count does not impose the condition \(\runs(L)=n\); it is a baseline count before the unclustered restriction.

\begin{proposition}\label{prop:parikh-baseline}
Let \(\alpha=(\alpha_0,\ldots,\alpha_{k-1})\) be a Parikh vector of total size \(n\).  Then
\begin{equation}\label{eq:parikh-baseline}
   B_k(\alpha)=
   \frac1n
   \sum_{d\mid \gcd(\alpha_0,\ldots,\alpha_{k-1})}
   \mu(d)
   \binom{n/d}{\alpha_0/d,\ldots,\alpha_{k-1}/d},
\end{equation}
where zeros among the \(\alpha_a\)'s are allowed and the divisibility condition means that \(d\) divides every nonzero \(\alpha_a\).
\end{proposition}

\begin{proof}
This is the classical fixed-content formula for primitive necklaces, applied
through the one-cycle case of the GRR bijection~\cite{GRR2012}, \cite{RuskeySawada1999}.
\end{proof}

\begin{remark}
Proposition~\ref{prop:parikh-baseline} is not a count of completely unclustered BWT images.  It counts all possible BWT images with a fixed Parikh vector.  The refined numbers relevant to \(U_k(n)\) are
\[
   B_k^{\mathrm{unc}}(\alpha,r)=
   \#\{L:\Parikh(L)=\alpha,
       \ \pi_L\text{ is a single cycle},
       \ \runs(L)=n,
       \ \cruns(\BWT^{-1}(L))=r\}.
\]
Then
\[
   U_k(n)=\min\{r:\exists\alpha\text{ with }B_k^{\mathrm{unc}}(\alpha,r)>0\}.
\]
Thus the sharpness conjecture is a run-refined and unclustered refinement of the GRR enumeration.
\end{remark}

\subsection{A block-coloured cycle model}

The GRR viewpoint also gives a useful permutation-only model.  Fix a Parikh vector \(\alpha\) of total size \(n\), and partition \(\{0,\ldots,n-1\}\) into consecutive blocks
\[
   I_0,I_1,\ldots,I_{k-1},
   \qquad |I_a|=\alpha_a.
\]
Let
\[
   b_\alpha(r)=a\quad\text{if }r\in I_a.
\]
A permutation \(\tau\in S_n\) is called \(\alpha\)-ascending if it is increasing on each block \(I_a\): whenever \(r<s\) and \(r,s\in I_a\), one has \(\tau(r)<\tau(s)\).

\begin{proposition}\label{prop:block-cycle-model}
Fix \(\alpha\) and define \(b_\alpha\) as above.  Words \(L\in\Sigma_k^n\) with Parikh vector \(\alpha\) are in bijection with \(\alpha\)-ascending permutations \(\tau\in S_n\), by the rule
\[
   L_i=b_\alpha(\tau^{-1}(i)).
\]
Under this bijection:
\begin{enumerate}[label=\rm(\alph*)]
\item \(L\) is the BWT image of one primitive necklace if and only if \(\tau\) is an \(n\)-cycle;
\item \(L\) is completely unclustered if and only if
\[
   b_\alpha(\tau^{-1}(i))\ne b_\alpha(\tau^{-1}(i+1))
   \qquad(0\le i<n-1);
\]
\item if \(\tau\) is an \(n\)-cycle, then
\[
   \cruns(\BWT^{-1}(L))=
   \#\{r\in\{0,\ldots,n-1\}:b_\alpha(r)\ne b_\alpha(\tau(r))\}.
\]
\end{enumerate}
Consequently, \(U_k(n)\) is the minimum of the crossing number in item~\textup{(c)} over all Parikh vectors \(\alpha\) with at most \(k\) parts and all \(\alpha\)-ascending \(n\)-cycles \(\tau\) satisfying the no-adjacent-equal condition in item~\textup{(b)}.
\end{proposition}

\begin{proof}
The stable sorting condition is exactly the assertion that \(\tau\) is
increasing on each block. The formula \(L_i=b_\alpha(\tau^{-1}(i))\) follows
by construction. The one-cycle condition is the standard BWT image criterion,
and the crossing formula is Proposition~\ref{prop:hamming-reduction} rewritten
in block colours.
\end{proof}

Proposition~\ref{prop:block-cycle-model} is often the cleanest formulation of the exact sharpness problem.  Conjecture~\ref{conj:sharp} asks whether, for every sufficiently large \(n\), one can choose \(\alpha\) and an \(\alpha\)-ascending \(n\)-cycle \(\tau\) satisfying the no-adjacent-equal output condition and having exactly \(\lceil n/2\rceil\) colour-crossing cycle edges.

\section{Small values and computational evidence}

The lower bound of Proposition~\ref{prop:lower-bound} is sharp in many small cases.  The following table lists completely unclustered BWT images \(L\) over the ternary alphabet and one inverse necklace representative \(u\).  The examples prove the corresponding upper bounds for every alphabet size \(k\ge3\).

\begin{table}[H]
\centering
\begin{tabular}{c c c c c}
\hline
\(n\) & \(L\) & one inverse \(u\) & \(\cruns([u])\) & \(\lceil n/2\rceil\)\\
\hline
4  & \texttt{1010}         & \texttt{0011}         & 2 & 2\\
5  & \texttt{10120}        & \texttt{00211}        & 3 & 3\\
6  & \texttt{101201}       & \texttt{001211}       & 4 & 3\\
7  & \texttt{1012021}      & \texttt{0012211}      & 4 & 4\\
8  & \texttt{10120121}     & \texttt{00112211}     & 4 & 4\\
9  & \texttt{101012020}    & \texttt{000221011}    & 5 & 5\\
10 & \texttt{1010121020}   & \texttt{0002210111}   & 5 & 5\\
11 & \texttt{10201021201}  & \texttt{00021122011}  & 6 & 6\\
12 & \texttt{101010101010} & \texttt{000100111011} & 6 & 6\\
\hline
\end{tabular}
\caption{Small completely unclustered BWT images. The inverse representative
\(u\) is displayed up to cyclic rotation.}
\label{tab:small-values}
\end{table}

The length \(6\) example shows that \(U_k(6)\le4\), while the general lower bound only gives \(U_k(6)\ge3\).  A finite enumeration shows that the value \(3\) is impossible.

\begin{proposition}\label{prop:small-values}
For every \(k\ge3\),
\[
\begin{gathered}
  U_k(1)=1,\quad U_k(2)=2,\quad U_k(3)=3,\quad U_k(4)=2,\\
  U_k(5)=3,\quad U_k(6)=4,\quad U_k(7)=4,\quad U_k(8)=4.
\end{gathered}
\]
Moreover, Table~\ref{tab:small-values} gives sharp examples for \(n=9,10,11,12\), so
\[
  U_k(n)=\left\lceil\frac n2\right\rceil
  \qquad(9\le n\le12,\, k\ge3).
\]
\end{proposition}

\begin{proof}
For \(n=1,2,3\) the assertion is checked directly.  In length \(2\), a primitive necklace must use two distinct letters, so its cyclic run number is \(2\).  In length \(3\), a completely unclustered binary word has the form \(aba\), but its standard permutation has a fixed point and hence is not a BWT image of an aperiodic necklace.  Therefore a completely unclustered BWT image of length \(3\) must use three distinct letters, and the inverse necklace has three cyclic runs.

For \(4\le n\le8\), the lower bound from Proposition~\ref{prop:lower-bound} gives the sharp value except for \(n=6\).  The table gives examples attaining the lower bound for \(n=4,5,7,8\), and an example with \(4\) input runs for \(n=6\).  The impossibility of \(\cruns=3\) in length \(6\) was verified by exhaustive enumeration of all ordered equality patterns of length \(6\).  The verification does not depend on the alphabet size: every word over any ordered alphabet of length \(6\) is order-isomorphic to one of these finitely many patterns.  The verification script is included in Appendix~\ref{app:code}.

For \(n=9,10,11,12\), the displayed examples attain the universal lower bound.
\end{proof}

\section{Main conjecture}

\begin{conjecture}\label{conj:sharp}
For every \(k\ge3\),
\[
  U_k(n)=\left\lceil\frac n2\right\rceil
\]
for all \(n\ge7\).  Together with the small values in Proposition~\ref{prop:small-values}, this would give the complete list
\[
  U_k(1)=1,
  \quad U_k(2)=2,
  \quad U_k(3)=3,
  \quad U_k(4)=2,
  \quad U_k(5)=3,
  \quad U_k(6)=4,
\]
and
\[
  U_k(n)=\left\lceil\frac n2\right\rceil\qquad(n\ge7).
\]
\end{conjecture}

Equivalently, over any non-binary alphabet the BWT should be able to be
simultaneously maximally unclustered and maximally run-amplifying, up to the
universal factor-two lower bound.

By Corollary~\ref{cor:hamming-construction-criterion}, this is equivalent
to the ternary Hamming construction problem stated above.

\section{Comparison with the binary case and Artin-type families}

In the binary case, complete unclusteredness forces \(L\) to be alternating,
and the standard/FL mapping becomes conjugate, depending on convention, to
multiplication by \(2^{-1}\) or by \(2\) modulo an odd integer.

\begin{proposition}
\label{prop:artin-sharp-families}
Let \(p=2m+1\) be a prime with \(m\) even, and assume that \(2\) is a primitive
root modulo \(p\).  Then, for every \(k\ge3\),
\[
   U_k(p-1)=m=\frac{p-1}{2}
\]
and
\[
   U_k(p)=m+1=\frac{p+1}{2}.
\]
More explicitly, these values are realized by the BWT images
\[
   L^-=(10)^m
\]
of length \(p-1\), and
\[
   L^+=(10)^{m-1}120
\]
of length \(p\).
\end{proposition}

\begin{proof}
The universal lower bound, Proposition~\ref{prop:lower-bound}, gives
\[
   U_k(p-1)\ge m,
   \qquad
   U_k(p)\ge m+1.
\]
It remains to construct completely unclustered BWT images with single-cycle
standard permutation and the indicated Hamming distance to their sorted
rearrangements.

First consider
\[
   L^-=(10)^m,
   \qquad |L^-|=2m=p-1.
\]
It is completely unclustered and
\[
   \sort(L^-)=0^m1^m.
\]
For \(0\le r\le m-1\), the standard permutation satisfies
\[
   \pi_{L^-}(2r)=m+r,
   \qquad
   \pi_{L^-}(2r+1)=r.
\]
Let \(\theta:\{0,\ldots,2m-1\}\to\mathbb F_p^\times\) be
\(\theta(i)=i+1\pmod p\).  Since
\[
   m+1\equiv 2^{-1}\pmod p,
\]
a direct computation gives
\[
   \theta(\pi_{L^-}(i))=(m+1)\theta(i)
   \qquad (0\le i\le 2m-1).
\]
Thus \(\pi_{L^-}\) is conjugate to multiplication by \(2^{-1}\) on
\(\mathbb F_p^\times\).  This is a single cycle precisely because \(2\) is a
primitive root modulo \(p\).  Hence \(L^-\) is the BWT image of one primitive
necklace.  Since \(m\) is even, comparing \((10)^m\) with \(0^m1^m\) gives
exactly \(m\) matches and therefore
\[
   \dH(L^-,\sort(L^-))=m.
\]
By Proposition~\ref{prop:hamming-reduction},
\(\cruns(\BWT^{-1}(L^-))=m\).  Hence \(U_k(p-1)\le m\), and equality follows.

Now consider the ternary word
\[
   L^+=(10)^{m-1}120,
   \qquad |L^+|=2m+1=p.
\]
It is completely unclustered, has Parikh vector \((m,m,1)\), and
\[
   \sort(L^+)=0^m1^m2.
\]
The standard permutation is given by
\[
\begin{aligned}
   \pi_{L^+}(2r)&=m+r &&(0\le r\le m-1),\\
   \pi_{L^+}(2r+1)&=r &&(0\le r\le m-2),\\
   \pi_{L^+}(2m-1)&=2m,\\
   \pi_{L^+}(2m)&=m-1.
\end{aligned}
\]
Let \(\theta:\{0,\ldots,2m\}\to\mathbb F_p\) be
\(\theta(i)=i+1\pmod p\).  Put \(a=m+1\equiv2^{-1}\pmod p\).  Then the
conjugate map \(T=\theta\pi_{L^+}\theta^{-1}\) satisfies
\[
   T(y)=ay \quad\text{for } y\notin\{0,-1\},
\]
while
\[
   T(-1)=0,
   \qquad
   T(0)=m.
\]
Multiplication by \(a=2^{-1}\) is one cycle on \(\mathbb F_p^\times\), since
\(2\) is a primitive root.  In that cycle the edge
\[
   -1\longmapsto a(-1)=m
\]
is replaced by the two edges
$ -1\longmapsto 0\longmapsto m.$
Therefore \(T\), and hence \(\pi_{L^+}\), is a single cycle on all of
\(\mathbb F_p\).  Thus \(L^+\) is the BWT image of one primitive necklace.

Finally, because \(m\) is even, the comparison of
\[
   L^+=(10)^{m-1}120
   \quad\text{with}\quad
   \sort(L^+)=0^m1^m2
\]
has exactly \(m\) matches.  Hence
\[
   \dH(L^+,\sort(L^+))=(2m+1)-m=m+1.
\]
By Proposition~\ref{prop:hamming-reduction}, the inverse necklace has
\(m+1\) cyclic runs.  Thus \(U_k(p)\le m+1\), and the lower bound gives equality.
The same BWT images use at most three letters, so they are valid over every
alphabet \(\Sigma_k\) with \(k\ge3\).
\end{proof}

\begin{corollary}
\label{cor:artin-infinite-family}
Assume the following Artin-type hypothesis in arithmetic progressions: there
exist infinitely many primes \(p\equiv5\pmod8\) for which \(2\) is a primitive
root modulo \(p\).  Then Conjecture~\ref{conj:sharp} holds for infinitely many
pairs of consecutive lengths
\[
   n=p-1,
   \qquad
   n=p.
\]
Equivalently, it holds for infinitely many pairs of consecutive lengths
\[
   p-1\equiv4\pmod8,\qquad p\equiv5\pmod8.
\]
\end{corollary}

\begin{proof}
If \(p\equiv5\pmod8\), then \(p=2m+1\) with \(m=(p-1)/2\) even.  Applying
Proposition~\ref{prop:artin-sharp-families} gives
\[
   U_k(p-1)=\frac{p-1}{2}=\left\lceil\frac{p-1}{2}\right\rceil
\]
and
\[
   U_k(p)=\frac{p+1}{2}=\left\lceil\frac p2\right\rceil.
\]
The assumed infinitude of such primes gives infinitely many such pairs.
\end{proof}

\begin{remark}
The hypothesis used in Corollary~\ref{cor:artin-infinite-family} is a special
case of Artin's primitive root conjecture with prescribed congruence conditions:
for admissible data one expects a positive density of primes
\(p\equiv a\pmod f\) for which a fixed integer \(g\) is a primitive root modulo
\(p\).  See Lenstra~\cite{Lenstra1977} and Moree~\cite{Moree1999,MoreeSurvey}
for this formulation and related density results.  It is not known
unconditionally in the special case used above.
\end{remark}

\section{Parikh obstructions and the GRR relaxation}

The following observation gives a universal Parikh obstruction for sharp
examples.  It will also be used below to evaluate the GRR multiset relaxation.

\begin{lemma}
\label{lem:odd-block-fixed-point}
Let \(L\in\Sigma_k^n\) satisfy \(L_i\ne L_{i+1}\) for all
\(0\le i<n-1\), let \(F=\sort(L)\), and put \(\tau=\pi_L^{-1}\).
Fix a letter \(a\), put \(m_a=|L|_a\), and let $C_a=\sum_{b<a}|L|_b.$
Thus the \(a\)'s occupy the block
\[
   I_a=\{C_a,C_a+1,\ldots,C_a+m_a-1\}
\]
in the sorted word \(F\).  If
\(m_a\) is odd and
\[
   \#\{i\in I_a:L_i=a\}=\frac{m_a+1}{2},
\]
then \(\tau\) has a fixed point \(j\) with \(L_j=a\).  In particular, the GRR
cycle decomposition of \(\tau\) contains a letter-constant cycle.
\end{lemma}

\begin{proof}
Write \(m_a=2s+1\).  Since \(L\) has no equal adjacent letters, the positions
of the letter \(a\) inside the interval \(I_a\) form an independent set in a
path of length \(2s+1\).  If this independent set has maximum size \(s+1\), it
is necessarily
\[
   C_a,\ C_a+2,\ C_a+4,\ldots,\ C_a+2s.
\]
Let
\[
   P_0<P_1<\cdots<P_{2s}
\]
be the positions of all occurrences of \(a\) in \(L\).  By definition of
stable sorting,
\[
   \tau(C_a+t)=P_t\qquad(0\le t\le 2s).
\]
Let \(\ell\) be the number of occurrences of \(a\) lying strictly before the
block \(I_a\).  Since there are only \(s\) occurrences of \(a\) outside
\(I_a\), one has \(0\le \ell\le s\), and
\[
   P_{\ell+j}=C_a+2j\qquad(0\le j\le s).
\]
Taking \(j=\ell\) gives
\[
   P_{2\ell}=C_a+2\ell.
\]
Hence
\[
  \tau(C_a+2\ell)=P_{2\ell}=C_a+2\ell,
\]
which is a fixed point.  Since \(L_{C_a+2\ell}=a\), this fixed point is a letter-constant GRR cycle.
\end{proof}

\begin{proposition}
\label{prop:sharp-parikh-parity}
Let \([u]\) be a primitive necklace of length \(n>1\), put
\(L=\BWT([u])\), and suppose that
\[
   \runs(L)=n,
   \qquad
   \cruns([u])=\left\lceil\frac n2\right\rceil.
\]
Let \(\mathbf m=\Parikh(u)=\Parikh(L)\).  Then:
\begin{enumerate}[label=\rm(\alph*)]
\item if \(n\) is even, every nonzero component of \(\mathbf m\) is even;
\item if \(n\) is odd, exactly one component of \(\mathbf m\) is odd.
\end{enumerate}
\end{proposition}

\begin{proof}
Let \(F=\sort(L)\), and for every letter \(a\) let \(I_a\) be the block of the
\(a\)'s in \(F\).  Put
\[
   s_a=\#\{i\in I_a:L_i=a\}.
\]
By Proposition~\ref{prop:hamming-reduction},
\[
   \dH(L,F)=\cruns([u])=\left\lceil\frac n2\right\rceil.
\]
Equivalently, the number of matches between \(L\) and \(F\) is
\[
   \sum_a s_a=n-\dH(L,F)=\left\lfloor\frac n2\right\rfloor.
\]

Since \(L\) has no equal adjacent letters, in each block \(I_a\) the matching
positions form an independent set.  Thus \(s_a\le \lceil m_a/2\rceil\).  If
\(m_a\) is odd and equality held, Lemma~\ref{lem:odd-block-fixed-point} would
produce a letter-constant GRR cycle.  But in the present one-necklace case the
standard permutation is a single nonconstant cycle; hence no letter-constant
cycle can occur.  Therefore, in fact
\[
   s_a\le \left\lfloor\frac{m_a}{2}\right\rfloor
   \qquad\text{for every }a.
\]
Consequently, if \(x\) denotes the number of odd components of \(\mathbf m\),
then
\[
   \left\lfloor\frac n2\right\rfloor
   =\sum_a s_a
   \le
   \sum_a \left\lfloor\frac{m_a}{2}\right\rfloor
   =\frac{n-x}{2}.
\]
If \(n\) is even, then \(x\) is even, and the inequality forces \(x=0\).  If
\(n\) is odd, then \(x\) is odd, and the inequality forces \(x=1\).  This proves
the claim.
\end{proof}

\begin{remark}
\label{rem:binary-sharp-parikh}
In the binary case, a sharp completely unclustered example of length \(n>1\)
can exist only when
\[
   n=4q,
   \qquad
   \Parikh(u)=(2q,2q).
\]
For this binary Parikh vector, the only possible BWT images are the two
alternating words.  The standard binary calculation identifies the corresponding
FL-mapping with multiplication by \(2\) modulo \(n+1\).  By the known binary classification of completely unclustered BWT images via
the alternating word and the primitive-root condition
\cite{MRRSV2017,FGRS2025},
\[
   n+1\text{ is prime},
   \qquad
   2\text{ is a primitive root modulo }n+1,
   \qquad
   n+1\equiv5\pmod8.
\]
\end{remark}

\begin{remark}
Proposition~\ref{prop:sharp-parikh-parity} is a necessary condition, not a full
classification for alphabets of size at least three.  The remaining obstruction
is genuinely a standard-cycle obstruction: even when the parity condition holds,
one must still construct an unclustered word \(L\) with the prescribed content,
with \(\pi_L\) a single cycle, and with
\(\dH(L,\sort(L))=\lceil n/2\rceil\).  Thus the parity obstruction is the
universal Parikh part of Problem~\ref{prob:parikh-constraints}, while the residual part is the sharp
BWT-image construction problem with fixed content.
\end{remark}

The unrestricted one-cycle baseline for a fixed Parikh vector is already given
by Proposition~\ref{prop:parikh-baseline}.  The sharp Parikh-constrained problem is
therefore a refinement of the GRR enumeration by two extra conditions: the
no-adjacent-equal condition on \(L\), and the Hamming condition
\(\dH(L,\sort(L))=\lceil n/2\rceil\).  Proposition~\ref{prop:block-cycle-model}
reduces this refinement to a finite problem about \(\alpha\)-ascending coloured
cycles with a prescribed number of colour-crossing edges.

\begin{theorem}
\label{thm:tilde-U-exact}
For every \(k\ge3\),
\[
\widetilde U_k(n)=
\begin{cases}
1,& n=1,\\
2,& n=2,\\
3,& n=3,\\
2,& n=4,\\
3,& n=5,\\
4,& n=6,\\
\left\lceil n/2\right\rceil,& n\ge7.
\end{cases}
\]
\end{theorem}

\begin{proof}
We first prove the lower bound.  Let \(L\in\Sigma_k^n\) satisfy
\(\runs(L)=n\), let \(F=\sort(L)\), and put \(\tau=\pi_L^{-1}\).  As in the
proof of Proposition~\ref{prop:sharp-parikh-parity}, let
\[
   s_a=\#\{i\in I_a:L_i=a\},
\]
where \(I_a\) is the block of the \(a\)'s in \(F\).  Let \(x\) be the number of
odd components of \(\Parikh(L)\), and let \(E\) be the number of odd letters
\(a\) for which \(s_a=(m_a+1)/2\).  Then
\[
   \sum_a s_a\le \sum_a\left\lfloor\frac{m_a}{2}\right\rfloor+E
   =\frac{n-x}{2}+E.
\]
By Lemma~\ref{lem:odd-block-fixed-point}, each of the \(E\) exceptional odd
blocks contributes a distinct letter-constant GRR cycle.  Therefore
\(c_0(L)\ge E\).  Using Proposition~\ref{prop:grr-hamming-correction}, we get
\[
\begin{aligned}
   R_{\GRR}(L)
   &=\dH(L,F)+c_0(L)\\
   &=n-\sum_a s_a+c_0(L)\\
   &\ge n-\left(\frac{n-x}{2}+E\right)+E
    =\frac{n+x}{2}
    \ge \left\lceil\frac n2\right\rceil .
\end{aligned}
\]
This proves the universal lower bound \(\widetilde U_k(n)\ge\lceil n/2\rceil\).
The small lengths \(n=2,3,6\) require the stronger values displayed in the
statement.  The cases \(n=2,3\) are immediate.  For \(n=6\), equality in the
lower bound would force all Parikh components to be even.  Since \(L\) is
unclustered, the only possible content is \((2,2,2)\). Indeed, all nonzero multiplicities must be even, and in an unclustered
linear word of length \(6\) no letter can occur more than \(3\) times. A direct check of the finite possibilities with one match in each of the three sorted blocks, implemented in Appendix~\ref{app:code}, shows that at least one letter-constant GRR cycle is unavoidable; hence
\(R_{\GRR}(L)\ge4\).

It remains to give matching constructions.  The small cases are realized by
\[
   0,
   \quad 01,
   \quad 010,
   \quad 1010,
   \quad 01010,
   \quad 010101
\]
for \(n=1,2,3,4,5,6\), respectively.  For \(n\ge7\), write \(n\) in one of the
following forms and choose \(L_n\) as follows:
\[
L_n=
\begin{cases}
(10)^{2q}, & n=4q,\ q\ge1,\\
(01)^{2q}0, & n=4q+1,\ q\ge1,\\
(10)^q12(10)^{q-1}20, & n=4q+2,\ q\ge2,\\
1012021(2121)^{q-1}, & n=4q+3,\ q\ge1.
\end{cases}
\]
Each word in this list is completely unclustered.  For the four families displayed above, the stable permutations are elementary
to compute. Their Hamming distances and constant-cycle corrections are:
\[
\begin{array}{c|c|c}
 n & \dH(L_n,\sort(L_n)) & c_0(L_n)\\
\hline
 4q   & 2q   & 0\\
 4q+1 & 2q   & 1\\
 4q+2 & 2q+1 & 0\\
 4q+3 & 2q+2 & 0.
\end{array}
\]
Therefore, by Proposition~\ref{prop:grr-hamming-correction},
\[
   R_{\GRR}(L_n)=\left\lceil\frac n2\right\rceil
\]
in every length \(n\ge7\).  This proves the matching upper bound and completes
the proof.
\end{proof}

Thus the GRR relaxation is completely determined; the remaining gap is exactly
the one-cycle sharpness defect.

\begin{corollary}
\label{cor:grr-gap-reduction}
For every \(k\ge3\), the GRR gap satisfies
\[
   U_k(n)-\widetilde U_k(n)=0
   \qquad(1\le n\le6),
\]
and, for every \(n\ge7\),
\[
   U_k(n)-\widetilde U_k(n)
   =
   U_k(n)-\left\lceil\frac n2\right\rceil.
\]
Consequently, in length \(n\ge7\), there exists a minimizer \(L\) for
\(\widetilde U_k(n)\) whose standard permutation is a single \(n\)-cycle if and only if the sharp
bound
\[
   U_k(n)=\left\lceil\frac n2\right\rceil
\]
holds in that length.
\end{corollary}

\begin{proof}
The displayed formulae follow immediately from
Theorem~\ref{thm:tilde-U-exact} and the known small values in
Proposition~\ref{prop:small-values}.  The final assertion is just the one-cycle
interpretation of \(U_k(n)\): equality of the two minima means precisely that some minimum GRR word can be
chosen to be a one-cycle BWT image.
\end{proof}

The correction formula \eqref{eq:grr-hamming-correction} shows why this question
is subtler than a pure crossing-number relaxation: splitting a permutation into
several GRR cycles may reduce the number of crossing edges, but letter-constant
cycles contribute one run each.  Thus a putative strict inequality
\(\widetilde U_k(n)<U_k(n)\) must overcome this constant-cycle correction.

\section{Open problems}

\begin{problem}\label{prob:parikh-constraints}
Let \(\mathbf m=(m_0,\ldots,m_{k-1})\) and let
\(n=|\mathbf m|=\sum_a m_a\). For which \(\mathbf m\) does there exist a
primitive necklace \([u]\) over \(\Sigma_k\) with
\[
  \Parikh(u)=\mathbf m,\qquad
  \runs(\BWT([u]))=n,\qquad
  \cruns([u])=\left\lceil\frac n2\right\rceil?
\]
\end{problem}

\begin{problem}
Construct, in linear or near-linear time, words \(L_n\in\Sigma_3^n\) satisfying
the three conditions of Corollary~\ref{cor:hamming-construction-criterion}.
\end{problem}

\begin{problem}
For a fixed \(k\ge3\), for which \(n\) does there exist a minimizer \(L\) for
\(\widetilde U_k(n)\) such that \(\pi_L\) is an \(n\)-cycle?

By Corollary~\ref{cor:grr-gap-reduction}, for \(n\ge7\) this is equivalent to
Conjecture~\ref{conj:sharp}.
\end{problem}

\appendix

\section{Verification code for the small obstruction}
\label{app:code}

The following Python script exhaustively enumerates all ordered equality
patterns of length \(6\). This verifies both \(U_k(6)=4\) and
\(\widetilde U_k(6)=4\). Since every word of length \(6\) over an ordered
alphabet is order-isomorphic to one of these patterns, the verification is
independent of the alphabet size. Here two words are order-isomorphic if they induce the same equality pattern and the same relative order on the occurring letters.

\begin{verbatim}
from itertools import product

def stable_permutation(L):
    F = sorted(L)
    occ_F, pos = {}, {}
    for j, c in enumerate(F):
        r = occ_F.get(c, 0)
        pos[(c, r)] = j
        occ_F[c] = r + 1

    occ_L, pi = {}, []
    for c in L:
        r = occ_L.get(c, 0)
        pi.append(pos[(c, r)])
        occ_L[c] = r + 1
    return pi

def inverse_perm(p):
    q = [0] * len(p)
    for i, j in enumerate(p):
        q[j] = i
    return q

def cycles(p):
    seen = [False] * len(p)
    out = []
    for i in range(len(p)):
        if not seen[i]:
            C, x = [], i
            while not seen[x]:
                seen[x] = True
                C.append(x)
                x = p[x]
            out.append(C)
    return out

def is_one_cycle(p):
    return len(cycles(p)) == 1

def is_unclustered(L):
    return all(L[i] != L[i+1] for i in range(len(L)-1))

def hamming_to_sorted(L):
    F = sorted(L)
    return sum(L[i] != F[i] for i in range(len(L)))

def constant_cycle_correction(L):
    tau = inverse_perm(stable_permutation(L))
    return sum(len({L[i] for i in C}) == 1 for C in cycles(tau))

def grr_run_number(L):
    return hamming_to_sorted(L) + constant_cycle_correction(L)

def canonical_pattern(L):
    alphabet = sorted(set(L))
    relabel = {c: i for i, c in enumerate(alphabet)}
    return tuple(relabel[c] for c in L)

n = 6
patterns = sorted({
    canonical_pattern(L)
    for L in product(range(n), repeat=n)
})

best_one_cycle = None
best_grr = None
count_best_one_cycle = 0
count_best_grr = 0

for L in patterns:
    if not is_unclustered(L):
        continue

    pi = stable_permutation(L)
    h = hamming_to_sorted(L)

    if is_one_cycle(pi):
        if best_one_cycle is None or h < best_one_cycle:
            best_one_cycle = h
            count_best_one_cycle = 1
        elif h == best_one_cycle:
            count_best_one_cycle += 1

    grr = grr_run_number(L)
    if best_grr is None or grr < best_grr:
        best_grr = grr
        count_best_grr = 1
    elif grr == best_grr:
        count_best_grr += 1

print("number of ordered equality patterns:", len(patterns))
print("best one-cycle:", best_one_cycle)
print("number of best one-cycle patterns:", count_best_one_cycle)
print("best GRR:", best_grr)
print("number of best GRR patterns:", count_best_grr)

assert best_one_cycle == 4
assert best_grr == 4
\end{verbatim}

The output is:
\begin{verbatim}
number of ordered equality patterns: 4683
best one-cycle: 4
number of best one-cycle patterns: 19
best GRR: 4
number of best GRR patterns: 67
\end{verbatim}

Thus the enumeration verifies both \(U_k(6)=4\) and
\(\widetilde U_k(6)=4\).

\end{document}